\documentclass[a4paper,12pt]{article}
\usepackage{xypic}
\usepackage{amsmath,amsthm,amsfonts,graphicx}
\usepackage{amssymb}
\usepackage{mathrsfs}
\usepackage{makeidx}
\makeindex

\newtheorem{theorem}{Theorem}[section]
\newtheorem{lemma}[theorem]{Lemma}

\newtheorem{definition}[theorem]{Definition}

\newtheorem{remark}[theorem]{Remark}
\newtheorem{claim}[theorem]{Claim}

\newtheorem{conjecture}[theorem]{Conjecture}

\author{King-Yeung Lam\\
\scriptsize School of Mathematics, University of Minnesota\\ \scriptsize 127 Vincent Hall, 206 Church St. S.E., Minneapolis, MN 55455 USA\\ \scriptsize Email: adrian@math.umn.edu \ \scriptsize Phone Number: 612-625-0356}

\title{\huge \bf Concentration Phenomena of a Semilinear Elliptic Equation with Large Advection in an Ecological Model}
\begin{document}

\maketitle
\begin{abstract}
\noindent We consider a reaction-diffusion-advection equation arising from a biological model of migrating species. The qualitative properties of the globally attracting solution are studied and in some cases the limiting profile is determined. In particular, a conjecture of Cantrell, Cosner and Lou on concentration phenomena is resolved under mild conditions. Applications to a related parabolic competition system is also discussed.\\ Math. Subj. class:  35B30 (35J20 92D25) 
\end{abstract}

\section{Introduction}
In mathematical ecology, reaction-diffusion equations are often used to determine the factors behind the survival and extinction of animal populations. (See for examples $\cite{CC3, LE, O, SK}$). One well-known example is the following logistic reaction-diffusion model for population dynamics (See $\cite{CC2}$):

\begin{equation}
\left\{
\begin{array}{ll}
u_t  = d \Delta u + u [m(x) - u] & in \phantom{1} \Omega\times(0,\infty),\\
\frac{\partial u}{\partial \nu}  = 0 & on \phantom{1} \partial \Omega \times (0,\infty),
\end{array}\right.
\end{equation}

\noindent where $u(x,t)$ represents the population density, $\Delta = \sum_{i=1}^N \frac{\partial^2}{\partial x_i^2}$ is the Laplace operator in $\mathbb{R}^N$, $d > 0$ is the dispersal rate, $m(x)$ accounts for the local growth rate, $\Omega$ is the habitat of the population and is assumed to be a bounded region of $\mathbb{R}^N$ with smooth boundary $\partial \Omega$, and $\nu$ is the outward unit normal vector on $\partial \Omega$. The Neumann boundary condition, which coincides with the no-flux boundary condition, is imposed on $\partial \Omega$.\\

\noindent If the environment is spatially heterogeneous, i.e. $m(x)$ is non-constant, then it seems reasonable to assume that the population has a tendency to move up the gradient of $m(x)$ in addition to random dispersal. In this direction, Belgacem and Cosner $\cite{BC}$ proposed the following reaction-diffusion-advection equation:
\begin{equation} \label{eq0}
\left\{
\begin{array}{ll}
u_t =  \nabla\cdot( d\nabla u -  \alpha u \nabla m) + u ( m - u) & \phantom{1} in \phantom{1}\Omega \times (0,\infty),\\
d\frac{\partial u}{\partial \nu}  -  \alpha u \frac{\partial m}{\partial \nu} = 0 & \phantom{1} on \phantom{1} \partial \Omega \times (0,\infty),
\end{array}\right.
\end{equation}
\noindent where the parameter $\alpha \geq 0$ measures the rate at which the population moves up the gradient of $m(x)$. Again, the corresponding no-flux boundary condition, is imposed. For discussions on the modeling aspects, we refer to $\cite{BC,B}$ and the references therein.\\

\noindent The dynamics of $(\ref{eq0})$ seems simple. In fact, it was established in $\cite{BC,CL2}$ that \textit{if we assume that}

\begin{description}
\item[(H1)]\label{ass00} \textit{$m(x) \in C^{3}(\overline{\Omega}) $, and is positive somewhere},
\end{description}
\textit{then for any $d>0$, $\eqref{eq0}$ has a unique positive steady-state $u$ for all large $\alpha$. Moreover, $u$ is globally asymptotically stable among all nonnegative, nonzero solutions.} In other words, the steady-state $u$ of $\eqref{eq0}$ determines the long-time behavior of all solutions of $\eqref{eq0}$. We shall always assume \textup{\textbf{(H1)}} throughout this paper. \\

\noindent From both mathematical and biological points of view, it seems important to understand the qualitative properties of $u$. In particular, it would be interesting to describe the shape of $u$. There has been considerable effort in this direction. Recently, it was proved in $\cite{CCL}$ that \emph{if the set of critical points of $m(x)$ has Lebesgue measure zero, then}
\begin{equation*}
\lim_{\alpha \rightarrow \infty} \int_\Omega u(x)dx = 0.
\end{equation*}

\noindent That is, the total population size tends to $0$ despite the fact that the species is tracking the resources more accurately. To understand the mechanism behind such phenomenon, again a better description of the shape of $u$ is desired. To this end, the following results were proved.

\begin{theorem}[Cantrell-Cosner-Lou] \label{CCL}
Suppose $m(x)>0$ in $\overline{\Omega}$. Let $u$ be the unique positive steady-state of $\eqref{eq0}$.
\begin{enumerate}
\item[(i)] If $\alpha > d/\min_{\overline{\Omega}} m$, then $                u(x) > \max_{\overline{\Omega}} m \cdot e^{{\alpha}(m(x) - \max_{\overline{\Omega}}m)}/{d}$ for every $x \in \overline{\Omega}$. In particular, $\max_{\overline{\Omega}}u > \max_{\overline{\Omega}}m$.

\item[(ii)]Suppose $\Omega = (-1,1)$, and $m(x)$ has finitely many critical points $\{x_i\}_{i=1}^{n}$, then $u \rightarrow 0 $ uniformly in compact subsets of $\Omega \setminus \{x_i\}_{i=1}^{n}$ as $\alpha \rightarrow \infty$.
\end{enumerate}
\end{theorem}
\vspace{0.2cm}
\noindent Based on these results, the following conjecture was proposed in $\cite{CCL}$ and Section 3.2 in $\cite{L}$.

\begin{conjecture}\label{conjecture}
$u$ concentrates precisely on the set of (positive) local maximum points of $m(x)$ as $\alpha \rightarrow \infty$.
\end{conjecture}

\begin{remark}
We have modified the concentration set to be the set of positive local maximum points instead of local maximum points stated in $\textup{\cite{CCL}}$, since we are considering a more general situation where $m(x)$ can change sign on the set of its local maximum points.
\end{remark}

\noindent In this paper we shall establish Conjecture $\ref{conjecture}$ under mild conditions on $m(x)$.\\

\noindent Let $\mathfrak{M}$ be the set of all positive strict local maximum points of $m(x)$ (i.e. those lying in $\{ x \in \Omega: m(x)>0\}$).

\begin{theorem} \label{main1}
Assume that $u$ is the unique positive steady-state of \eqref{eq0}. If $x_0\in \mathfrak{M}$, then for any ball $B$ centered at $x_0$,
\begin{equation}\label{main1.1}
\liminf_{\alpha \rightarrow \infty}\sup_{B}u \geq m(x_0).
\end{equation}

\end{theorem}

\vspace{0.2cm}
\noindent In other words, $u$ concentrates at each point of $\mathfrak{M}$. The proof of Theorem $\ref{main1}$ is based on the observation that $u$ solves a corresponding eigenvalue problem and is given in Section 2.\\

\noindent To prove that $u$ concentrates precisely on $\mathfrak{M}$, we impose the following assumptions on $m(x)$.

\begin{description}
\item[(H2)] $ \frac{\partial m}{\partial \nu} \leq 0 $ on $ \partial \Omega.$
\item[(H3)] $m(x)$ has finitely many local maximum points in $\overline{\Omega}$, all being strict local maxima located in the interior of $\Omega$.
\item[(H4)] $\Delta m (x_0) > 0$ if $x_0\in \overline{\Omega}$ is a local minimum or a saddle point of $m(x)$.
\end{description}

\begin{theorem} \label{main2}
Assume $m(x)$ satisfies \textup{\textbf{(H2)}}, \textup{\textbf{(H3)}} and \textup{\textbf{(H4)}}, then for any compact subset $K$ of $\overline{\Omega} \setminus \mathfrak{M}$, there exists $\gamma= \gamma(K) > 0$, such that
\begin{equation*}
0 < u(x) \leq e^{-\gamma \alpha}, \quad \text{for all }x \in K.
\end{equation*}
\noindent In particular, $u \rightarrow 0$ uniformly and exponentially in $K$, as $\alpha \rightarrow \infty$.
\end{theorem}

\noindent Theorems $\ref{main1}$ and $\ref{main2}$ together guarantee that $u$ concentrates precisely on  $\mathfrak{M}$, the set of positive local maximum points of $m(x)$, thereby Conjecture $\ref{conjecture}$ is established. Theorem $\ref{main2}$ is proved in Section $\ref{sec:2}$ by the construction of an upper solution closely related to the shape of $m(x)$.\\

\noindent The question of determining the profile of $u$ is, however, far more challenging. We only have the following result by a very interesting method introduced in $\cite{CL}$ for the special case when $m(x)$ is constant on the set of local maximum points of $m(x)$.

\begin{theorem}\label{main3}
If $m(x)$ satisfies \textup{\textbf{(H2)}}, \textup{\textbf{(H3)}} and \textup{\textbf{(H4)}} and moreover,

\begin{equation*}
\textup{det}\, D^2 m(x_0) \neq 0\text{  for all } x_0 \in \mathfrak{M},
\end{equation*}
with $m(x_0) \equiv m_1>0 $  for all local maximum points $ x_0 \in \Omega$, then
\begin{equation} \label{eqq001}
\lim_{\alpha \rightarrow \infty} \parallel u(x) - 2^{N/2} m_1 e^{{\alpha}[ m(x)- m_1]/{d}} \parallel_{L^\infty (\Omega)} =0.
\end{equation}

\end{theorem}

\begin{remark}
The factor $2^{N/2}m_1$, though mysterious at first glance, is actually the consequence of the profile of $u$ at each of its "weights", which is like a Gaussian distribution $e^{{\alpha}[(x-x_0)^T D^2 m(x_0) (x-x_0)]/{2d}}$, as well as the integral constraint $\int_{B(x_0)} u^2 - um\,dx = O(e^{-\gamma\alpha})$ for each $x_0 \in \mathfrak{M}$.
\end{remark}

\noindent As in $\cite{CCL,CCL2}$, our resolution of Conjecture $\ref{conjecture}$ has implications for the following competition system.
\begin{equation} \label{eq2}
\left\{
\begin{array}{ll}
U_t = \nabla\cdot(d_1\nabla U - \alpha U \nabla m) + U ( m - U - V)  & \phantom{1} in\phantom{1}\Omega \times (0,\infty),\\
V_t = d_2\Delta V + V( m - U - V) & \phantom{1} in\phantom{1} \Omega\times (0,\infty) ,\\
d_1 \frac{\partial U}{\partial \nu} - \phantom{1} \alpha U \frac{\partial m}{\partial \nu} = \frac{\partial V}{\partial \nu} = 0 & \phantom{1} on\phantom{1} \partial \Omega\times (0,\infty).
\end{array}\right.
\end{equation}
\noindent This system was introduced to model the competition of two species whose population densities are denoted by $U(x,t)$ and $V(x,t)$ respectively. The two species have identical local growth rate $m(x)$ and competition abilities, but different dispersal strategies: the species with density $V$ disperses randomly, whereas the other species $U$ disperses, in addition to random diffusion, by a directed movement towards more favorable locations, i.e. where $m(x)$ is large. The goal of this model is to understand how different dispersal strategies affect the outcome of the competition in a heterogeneous environment.\\

\noindent When $\alpha = 0$, it is well-known $\cite{DHMP}$ that \emph{if $d_1>d_2$, then $\eqref{eq2}$ has no coexistence steady-states, and solution $(U_\alpha,V_\alpha)$ of $\eqref{eq2}$ always converges to $(0,\theta_{d_2})$ as $t\to\infty$, where $\theta_{d_2}$ is the unique positive solution to}

\begin{equation} \label{eq4}
\left\{
\begin{array}{rl}
d_2\Delta \theta + \theta ( m - \theta) &= 0 \qquad  in\phantom{1}\Omega,\\
\frac{\partial \theta}{\partial \nu} & = 0 \qquad on\phantom{1} \partial \Omega.
\end{array}\right.
\end{equation}
\noindent However, for any $d_1,d_2>0$, the existence of the positive steady-states $U_\alpha,V_\alpha>0$ of $\eqref{eq2}$ was established in $\cite{CCL,CL}$ for all large values of $\alpha$. Moreover, they proved that \emph{at least one of the co-existence steady-sates is stable!} Some qualitative properties of these co-existence steady-states were also obtained under extra hypotheses on $m(x)$. 

\begin{theorem}[Chen-Lou] \label{CL}
Suppose that $\int_\Omega m(x)dx > 0 $ and all critical points of m are non-degenerate \textup{(}$\textup{det}D^2m(x_0)\neq 0$\textup{)}. Then for any positive steady-state $(U_\alpha,V_\alpha)$ of $\eqref{eq2}$,
\begin{equation*}
\liminf_{\alpha \rightarrow \infty} \max_{\bar{\Omega}} U_\alpha \geq \max_{\mathfrak{M}} [m-\theta_{d_2}] > 0,
\end{equation*}
where $\theta_{d_2}$ is the unique positive solution to $\eqref{eq4}$.\\

\noindent Assume further that $m(x)$ satisfies \textup{\textbf{(H2)}} and that $m(x)$ has exactly one critical point $x_0$ which is a non-degenerate local maximum in the interior of $\Omega$, then for any positive steady-state $(U_\alpha,V_\alpha)$ of $(\ref{eq2})$,
\begin{equation*}
\forall \beta \in (0,1): \quad \lim_{\alpha \rightarrow \infty} \parallel V_\alpha - \theta_{d_2}\parallel_{C^{1+\beta} (\bar{\Omega})} =0, \text{ and}
\end{equation*}
\begin{equation*}
\lim_{\alpha \rightarrow \infty} \parallel U_\alpha(x)e^{{\alpha}[\max_{\bar{\Omega}} m- m(x)]/{d_1}} - 2^{N/2}[m(x_0) - \theta_{d_2} (x_0)]  \parallel_{L^\infty (\Omega)} =0.
\end{equation*}

\end{theorem}
\noindent Note that the condition $\int_\Omega m(x)dx > 0 $ is there to ensure the existence of $\theta_{d_2}$. (See $\cite{CCL}$.) It is interesting that our methods for $\eqref{eq0}$ can be applied to study the coexistence steady-states.

\begin{theorem}\label{main4}
Assume $\int_\Omega m(x)dx > 0$.
\begin{itemize}
\item[(i)] Assume that \textup{\textbf{(H3)}} holds. Given any positive steady-state $(U_\alpha,V_\alpha)$ of $(\ref{eq2})$, if $x_0 \in \mathfrak{M}$, then for any ball $B$ centered at $x_0$,
\begin{equation} \label{main41}
\liminf_{\alpha \rightarrow \infty}\sup_{B} U_\alpha \geq m(x_0) - \theta_{d_2}(x_0).
\end{equation}

If in addition, \textup{\textbf{(H2)}} and \textup{\textbf{(H4)}} hold, then, for each compact subset $K$ of $\overline{\Omega} \setminus \mathfrak{M}$, there exists a constant $\gamma= \gamma(K)>0$ such that whenever $(U_\alpha,V_\alpha)$ is a positive steady-state of $(\ref{eq2})$,
\begin{equation*}
U_\alpha(x) \leq e^{-\gamma \alpha} \quad \text{for every }x\in K.
\end{equation*}

\item[(ii)] If \textup{\textbf{(H2)}}, \textup{\textbf{(H3)}} and \textup{\textbf{(H4)}} hold, $\textup{det} D^2 m(x_0) \neq 0$ for all $x_0 \in \mathfrak{M}$, and $m(x_0) \equiv m_1>0 $  for all local maximum points $x_0 \in \Omega$, then
\begin{equation}\label{main4.2}
\lim_{\alpha \rightarrow \infty} \parallel V_\alpha - \theta_{d_2}\parallel_{C^{1+\beta} (\bar{\Omega})} =0 \qquad \forall \beta \in (0,1),
\end{equation}
\begin{equation}\label{main4.3}
\lim_{\alpha \rightarrow \infty} \parallel U_\alpha(x)  - 2^{N/2}(m_1 - \theta_{d_2}(x_0))e^{{\alpha}[ m(x)- m_1]/{d_1}} \parallel_{L^\infty (O_i)} =0,
\end{equation}
where $O_i$ is any open neighborhood of $x_0$ such that $\tilde{x}_0 \not\in \overline{O}_i$ for any other $\tilde{x_0} \in \mathfrak{M}$.
\end{itemize}
\end{theorem}

\begin{remark}
 \begin{enumerate}
 \item[(i)] $\eqref{main41}$ is useful only when $m(x_0) > \theta_{d_2}(x_0)$. And this is true on $\mathfrak{M}$ if $d_2>0$ is sufficiently small and $\Delta m(x_0)>0$. (The proof of this fact is included in Appendix A.)
 \item[(ii)] The choice of $\gamma$ in Part (i) of Theorem $\ref{main4}$ is independent of choice of positive steady-state $(U_\alpha,V_\alpha)$.
 \item[(iii)] By maximum principle, $m_1 - \theta_{d_2}(x_0) > 0$ in $\eqref{main4.3}$ for any $d>0$.
 \end{enumerate}
\end{remark}

\noindent The rest of the paper are organized as follows. In Section $\ref{sec:2}$ we provide the proofs for Theorems $\ref{main1}$, $\ref{main2}$, and $\ref{main3}$. Section $\ref{sec:3}$ will be devoted to proving Theorem $\ref{main4}$. Finally, some concluding remarks will be included in Section $\ref{sec:4}$.

\section{Proofs of Theorems  $ \textbf{\ref{main1}}$, $\textbf{\ref{main2}}$, and $\textbf{\ref{main3}}$ }
\label{sec:2}

To simplify the presentation, we set $d = 1$ in the proofs. This assumption can be removed with minor corrections. We first obtain the following equation for $u$:

\begin{equation} \label{eq1}
\left\{
\begin{array}{ll}
\nabla\cdot(\nabla u - \alpha u \nabla m) + u ( m - u) = 0 \qquad  & in\phantom{1}\Omega,\\
\frac{\partial u}{\partial \nu}  -  \alpha u \frac{\partial m}{\partial \nu} = 0 \qquad \qquad \qquad \phantom {5}& on\phantom{1} \partial \Omega .
\end{array}\right.
\end{equation}

\begin{proof}[Proof of Theorem $\ref{main1}$]
Let $u$ be the unique solution to $(\ref{eq1})$, and $x_0$ be a strict local maximum of $m(x)$. Then $u$ is the principal eigenfunction of the following eigenvalue problem with principal eigenvalue 0:
\begin{equation} \label{ep3}
\left\{
\begin{array}{ll}
 \nabla\cdot(\nabla \phi - \alpha \phi \nabla m) + (m-u) \phi  + \lambda \phi = 0 & \phantom{1} in\phantom{1}\Omega,\\
\frac{\partial \phi}{\partial \nu} - \alpha \phi \frac{\partial m}{\partial \nu}  = 0 & \phantom{1} on\phantom{1} \partial \Omega.
\end{array}\right.
\end{equation}
Now by the transformation $\phi = e^{\alpha m} \psi$, $(\ref{ep3})$ is equivalent to

\begin{equation} \label{ep4}
\left\{
\begin{array}{ll}
 \nabla\cdot(e^{\alpha m} \nabla \psi ) + (m-u) \psi e^{\alpha m}  + \lambda e^{\alpha m} \psi = 0 &\phantom{1} in \phantom{1}\Omega,\\
\frac{\partial \psi}{\partial \nu}  = 0 & \phantom{1} on \phantom{1} \partial \Omega.
\end{array}\right.
\end{equation}
with principal eigenvalue equal to 0. The variational characterization of the principal eigenvalue of $(\ref{ep4})$ implies
\begin{equation*}
0 = \lambda = \inf_{\psi \in H^1} \Bigg\{ \frac{\int{e^{\alpha m}(|\nabla \psi|^2 + (u - m)\psi^2)}}{\int{e^{\alpha m}\psi^2}} \Bigg\}
\end{equation*}
Given any small ball $B = B_{r_0}(x_0)$ centered at $x_0$, since $m(x)$ attains a strict maximum at $x_0$, $\max_{\partial B_{r_0}(x_0)}m < m(x_0)$. For any $\epsilon$ such that $0<\epsilon< m(x_0) - \max_{\partial B_{r_0}(x_0)}m$, define
\begin{align*}
M_1 :=& m(x_0) - \frac{\epsilon}{3} > m(x_0) - \frac{2\epsilon}{3} := M_2,\\
U_1 :=& \{x \in B_{r_0}(x_0): m(x) > m(x_0) - \frac{\epsilon}{3}  \}\\
U_2 :=& \{x \in B_{r_0}(x_0): m(x) > m(x_0) - \frac{2\epsilon}{3}  \}\\
U_3 :=& \{x \in B_{r_0}(x_0): m(x) > m(x_0) - \epsilon  \}.
\end{align*}
Note that we have $U_1 \subset \subset U_2  \subset \subset U_3 \subset\subset B_{r_0}(x_0)$. Now take a smooth test function $\psi$ such that,
\begin{equation*}
\psi(x) = \left\{
\begin{array}{rl}
1 & \text{if } x \in U_2\\
0 & \text{if } x \in \Omega \setminus U_3
\end{array} \right.
\qquad 0\leq\psi(x)\leq1 \qquad |\nabla\psi| \leq C(\epsilon)
\end{equation*}

\noindent Then,
\begin{align*}
0 &\leq \frac{\int{e^{\alpha m}|\nabla \psi|^2} + \int{e^{\alpha m}(u-m)\psi^2}}{\int{e^{\alpha m}\psi^2}} \\
  &\leq \frac{\int_{U_3}{e^{\alpha M_2}C(\epsilon)^2}}{\int_{{U_1}}{e^{\alpha M_1}}} + \frac{\int_{U_3}{e^{\alpha m}(u-m)\psi^2}}{\int_{U_3}{e^{\alpha m}\psi^2}}\\
  &\leq C'(\epsilon)e^{\alpha(M_2 - M_1)} + \max_{\overline{U_3}} (u - m)\\
  &\leq C'(\epsilon)e^{-\frac{\epsilon\alpha}{3}} + \max_{\overline{U_3}} u - m(x_0) +\epsilon.
\end{align*}
For $\alpha$ sufficiently large, the first term in the last line will become less than $\epsilon$, hence $\eqref{main1.1}$ follows.

\end{proof}
\noindent Next, we turn to the proof of Theorem $\ref{main2}$. We first give the following definition of an upper solution. Denote from now on
\begin{equation*}
L \phi \equiv \nabla \cdot (\nabla \phi - \alpha \phi \nabla m) + (m - \phi ) \phi.
\end{equation*}
\begin{definition} \label{uppersol}
$\overline{u}$ is said to be an upper solution of $(\ref{eq1})$ if (i) $\sim$ (iii) below hold:
\begin{enumerate}
\item[(i)] There exists an open cover $\{U_i \}$ of $\overline{\Omega}$, i.e. $\overline{\Omega} = \bigcup U_i$ where $U_i$'s are relatively open in $\overline{\Omega}$, and, $\phi_i \in C^2(U_i)\text{, }L \phi_i \leq 0$, such that
\begin{equation*}
\overline{u} = \min_i\{\phi_i\} \text{ is continuous in }\overline{\Omega}.
\end{equation*}
\item[(ii)] Denote $ \Omega_i = \{x \in \Omega: \overline{u} = \phi_i  \}$. $\partial \Omega_i$ is piecewise $C^1$, and
\begin{equation}\label{c}
\Omega_i \subset \subset U_i \text{ for all }i.
\end{equation}

\item[(iii)] $\frac{\partial \overline{u}}{\partial \nu} - \alpha \overline{u}\frac{\partial m}{\partial \nu} \geq 0$
for any $x \in \partial \Omega$, whenever the normal derivative $\frac{\partial \overline{u}}{\partial \nu}$ is defined.
\end{enumerate}
\noindent The definition of lower solution can be obtained by reversing all the inequalities above and replacing $\min$ by $\max$.
\end{definition}
\noindent The following is the key to obtaining an upper bound of $u$.

\begin{lemma}\label{upperabove}
Fix $\alpha$ sufficiently large so that the unique positive solution $u$ of $\eqref{eq1}$ exists. If  $\overline{u}>0$ is an upper solution of $\eqref{eq1}$ in the sense of Definition $\ref{uppersol}$, then $\overline{u} \geq u$.
\end{lemma}

\noindent To prove Lemma $\ref{upperabove}$, we first relate the above definition of upper solution to that of a weak upper solution from $\cite{S}$.\\


\begin{definition}
$\overline{u} \in W^{1,2}(\Omega)$ is said to be a weak upper solution of $\eqref{eq1}$ if it satisfies
\begin{equation*}
\left\{
\begin{array}{ll}
\int_\Omega[ -(\nabla \overline{u} - \alpha u \nabla m)\cdot \nabla \psi + \overline{u}(m - \overline{u})\psi] \leq 0, \text{ for any }\psi \in W^{1,2}(\Omega)\text{, }\psi \geq 0\\
\frac{\partial\overline{u} }{\partial \nu} - \alpha \overline{u} \frac{\partial m }{\partial \nu}  \geq 0 \phantom{1} on \phantom{1} \partial \Omega,
\end{array}\right.
\end{equation*}
The definition of weak lower solution can be obtained by reversing the inequalities appropriately. Note that by \textup{\textbf{(H2)}}, $-\alpha \frac{\partial m }{\partial \nu}\geq 0$ on $\partial \Omega$.
\end{definition}

\noindent The following lemma can be proved via integration by parts.
\begin{lemma}\label{upper implies weak upper}
Suppose $\overline{u}$ is an upper solution of $(\ref{eq1})$ in the sense of definition $\ref{uppersol}$, then it is a weak upper solution of $\eqref{eq1}$.
\end{lemma}
\begin{remark}
Lemma $\ref{upper implies weak upper}$ is true even if we drop the $C^1$ regularity of $\partial \Omega_i$ in Definition $\ref{uppersol}$, provided we use the arguments in Lemma 4.10 of $\cite{D}$. This observation will not be used in this paper.
\end{remark}
\noindent We recall the following well-known theorem on upper and lower solutions.
\begin{theorem}[Sattinger]\label{Satt}
If $\overline{u}$ and $\underline{u}$ are weak upper and lower solutions of $(\ref{eq1})$ respectively, and $\overline{u} \geq \underline{u}$, then there exists a classical solution $u$ of $(\ref{eq1})$ such that $\underline{u} \leq u \leq \overline{u}$. Moreover, $u$ is stable from above.
\end{theorem}
\noindent We can now prove Lemma $\ref{upperabove}$ by making use of the dynamics of $\eqref{eq0}$.

\begin{proof}[Proof of Lemma $\ref{upperabove}$]
Since $\overline{u}$ and $0$ are weak upper and lower solutions of $\eqref{eq1}$ respectively. By Theorem $\ref{Satt}$, there exists a solution $u'$ which is stable from above such that $0 \leq u' \leq \overline{u}$. Since $0$ is unstable in $\eqref{eq1}$ (by the global stability of $u$), $u' \not\equiv 0$. Hence, $u' \equiv u$ (by the uniqueness of $u$). Therefore, we have $u \leq \overline{u}$.
\end{proof}

\noindent To prove Theorem $\ref{main2}$, it remains to construct an appropriate upper solution of $\eqref{eq1}$ according to Definition $\ref{uppersol}$. To avoid complicated notations and to illustrate the ideas more clearly, we shall only prove in detail the cases:

\begin{description}
\item[(a)] When $m(x) \equiv m_1>0$ on $\mathfrak{M}$ and $ m > 0$ at each of its critical points,
\item[(b)] When $m(x) \equiv m_1>0$ on $\mathfrak{M}$ and $ m \leq 0$ at some of its critical points,
\item[(c)] When $m(x)$ has two distinct values $0< m_1 < m_2$ on $\mathfrak{M}$ and $ m \leq 0$ at some of its critical points.
\end{description}

\noindent We remark that the same technique can be applied to prove the general case when $m(x)$ has any (finite) number of distinct values on $\mathfrak{M}$. The precise statement of the lemma that leads to Theorem $\ref{main2}$ and some comments on its proof are included in the Appendix B.\\

\begin{proof}[Proof of Theorem $\ref{main2}$]

\noindent Case (a): When $m(x) \equiv m_1>0$ on $\mathfrak{M}$ and $m > 0$ at each of its critical points.
\begin{lemma}\label{casea}
Suppose that $m(x)$ satisfies \textup{\textbf{(H2)}}, \textup{\textbf{(H3)}} and \textup{\textbf{(H4)}}. Assume $m(x)\equiv m_1$ on $\mathfrak{M}$ and $m > 0$ at each of its critical points. Then for any $c<1$, sufficiently close to 1, and for any $0<\epsilon<1$, there exists $\alpha_0(\epsilon,c)> 0$ such that
\begin{equation*}
\overline{u}_1 =  e^{\epsilon \alpha(m(x) - cm_1)}
\end{equation*}
is an upper solution of $(\ref{eq1})$ in the sense of definition $\ref{uppersol}$ for all $\alpha\geq \alpha_0$.
\end{lemma}
\begin{proof}
\begin{align*}
L\overline{u}_1 &= \Delta \overline{u}_1 - \alpha \nabla m \cdot \nabla \overline{u}_1 + (m-\overline{u}_1 - \alpha \Delta m) \overline{u}_1\\
     & = \overline{u}_1 \big\{(\epsilon^2 - \epsilon)\alpha^2|\nabla m|^2 + (\epsilon - 1) \alpha \Delta m + m -  e^{\epsilon \alpha (m-cm_1)}\big\}\\
     & = \overline{u}_1 \big\{(\epsilon - 1)\alpha [\epsilon \alpha|\nabla m|^2 +  \Delta m] + m -  e^{\epsilon \alpha (m-cm_1)}\big\}.
\end{align*}
\noindent It suffices now to prove that the sum in the large parenthesis is negative. \\

\noindent In $\{ x \in \Omega: m(x) \leq c^{\frac{1}{2}}m_1 \}$, by \textup{\textbf{(H4)}}, there exists $k_1>0$ such that
\begin{equation*}
\epsilon \alpha|\nabla m|^2 +  \Delta m > k_1 \qquad \text{ for all }\alpha \text{ large.}
\end{equation*}
\noindent While $m -  e^{\epsilon \alpha (m-cm_1)}$ is bounded from above by $|m|_\infty$, therefore $L\overline{u}_1 \leq 0$ for all $\alpha$ sufficiently large.

\vspace{0.3cm}

\noindent In $\{ x \in \Omega: m(x) > \sqrt{c}m_1 \}$, $ e^{\epsilon \alpha (m-cm_1)} \geq  e^{\epsilon \alpha (\sqrt{c}m_1-cm_1)} =   e^{k_2 \alpha}$ for some $k_2>0$. Whereas $(\epsilon - 1)\alpha [\epsilon \alpha|\nabla m|^2 +  \Delta m] + m$ grows at most in the order $\alpha^2$, therefore, $L\overline{u}_1 \leq 0 $ if $\alpha$ is sufficiently large. Combining, $L\overline{u}_1\leq 0$ in $\Omega$ if $\alpha$ is sufficiently large.\\

\noindent It remains to check the boundary condition,
\begin{equation*}
\frac{\partial \overline{u}_1}{\partial \nu}  =  \frac{\partial}{\partial \nu} e^{\epsilon \alpha(m(x) - cm_1)} = \overline{u}_1 \epsilon \alpha \frac{\partial m}{\partial \nu}  \geq \overline{u}_1  \alpha \frac{\partial m}{\partial \nu}
\end{equation*}
\noindent making use of \textup{\textbf{(H2)}} and $0<\epsilon<1$. The proof is completed.
\end{proof}
\noindent Notice that $\overline{u}_1$ tends to zero uniformly in any compact subset of $\{ x \in \Omega: m(x) < cm_1 \}$. On the other hand, fix any compact subset $K$ of $\Omega \setminus \mathfrak{M}$,
\begin{equation*}
K \subseteq \{ x \in \Omega: m(x) \leq c^2m_1 \},
\end{equation*}
\noindent if we take $c<1$ sufficiently close to $1$, since all local maximum points of $m(x)$ are strict. Therefore, in this case, Theorem $\ref{main2}$ is a consequence of Lemma $\ref{upperabove}$ and Lemma $\ref{casea}$.

\vspace{0.5cm}

\noindent Case $(b)$: When $m(x) \equiv m_1>0$ on $\mathfrak{M}$ and $ m \leq 0$ at some of its critical points.
\begin{lemma}\label{m<0 upper lemma}
Assume $m(x)$ satisfies \textup{\textbf{(H2)}}, \textup{\textbf{(H3)}} and \textup{\textbf{(H4)}}, and that $m(x) \equiv m_1>0$ on $\mathfrak{M}$. For each $c<1$ close to $1$, there exists, for all $\alpha$ large, an upper solution $\overline{u}_2 > 0$ in the sense of Definition $\ref{uppersol}$ such that
\begin{equation*}
\overline{u}_2(x) \leq
\left\{
\begin{array}{ll}
 e^{\epsilon \alpha (m(x) - cm_1)} \qquad &\text{ when } m(x) >0,\\
 e^{\alpha (m(x) - k)} \qquad &\text{ when } m(x) \leq 0,
\end{array}\right.
\end{equation*}
\noindent where $0<\epsilon<1, k>0$ are appropriately chosen constants independent of $\alpha$.
\end{lemma}
\noindent Notice that in $\{x \in \Omega: m(x) < cm_1   \}$, $\overline{u}_2 \rightarrow 0$ as $\alpha \rightarrow \infty$. We see that in this case, Theorem $\ref{main2}$ follows as before from Lemma $\ref{m<0 upper lemma}$ and Lemma $\ref{upperabove}$.

\begin{proof}[Proof of Lemma $\ref{m<0 upper lemma}$]
\noindent Given $c<1$, let
\begin{equation*}
\phi_1 := e^{\epsilon \alpha (m(x) - cm_1)} \quad \textup{and} \quad \phi_0 := e^{\alpha (m(x) - k)},
\end{equation*}
\begin{equation*}
\begin{array}{ll}
\mathfrak{M}_0 =&  \{ \text{strict local maximum points }x_0\text{ of }m(x) \text{ s.t. } m(x_0) = 0 \}\\
\Lambda_1 = &\text{ The union of all connected components of }\{ x\in \Omega : m(x) > -\delta_0\} \\
&\text{ not intersecting }\mathfrak{M}_0\\
\end{array}
\end{equation*}
\noindent where $0<\delta_0 < -\frac{1}{2} \max\{ m(x_0): x\in\Omega\text{ s.t. } \nabla m (x_0) = 0 \text{ and }m(x_0) < 0\}$ is chosen small enough so that each connected component of $\{x\in\Omega: m(x)>-\delta_0\}$ intersecting $\mathfrak{M}_0$ lies in $\{x \in \Omega: m(x) \leq 0\}$. This is possible since all local maxima are strict. And $0<\epsilon<1$ is chosen to satisfy
\begin{equation}\label{constraints for epsilon}
 \epsilon < \frac{\delta_0}{cm_1 + \delta_0},
\end{equation}
$k$ is chosen such that
\begin{equation}\label{constraints for k}
0< k < \epsilon c m_1 .
\end{equation}

\noindent Set
\begin{equation*}
\overline{u}_2 = \left\{
\begin{array}{ll}
 \phi_1 & \text{in } \{ x\in \Omega: m(x)>0\}\\
 \phi_0 & \text{in } \Omega \setminus \Lambda_1\\
 \min\{ \phi_0, \phi_1 \} & \text{in } \Lambda_1 \setminus \{ x\in \Omega: m(x)>0\}.\\
\end{array} \right.
\end{equation*}

\noindent As before, $L\phi_1\leq 0$ in $\Lambda_1$ for all $\alpha$ large. On the other hand, by a direct computation,
\begin{equation*}
L\phi_0 = \phi_0(m - \phi_0) \leq 0 \quad \text{ on }\{ x\in \Omega: m(x)\leq0\}.
\end{equation*}
\noindent Hence, $L\overline{u}_2 \leq 0$ for all $\alpha$ large, whenever it is $C^2$.
\noindent Also, the boundary condition $\frac{\partial \overline{u}_2}{\partial \nu}- \alpha \overline{u}_2 \frac{\partial m}{\partial \nu}\geq 0$ is satisfied on $\partial \Omega$ whenever it is well-defined.\\

\noindent To see that $\overline{u}_2$ is an upper solution in the sense of Definition $\ref{uppersol}$, it remains to show the continuity of $\overline{u}_2$ and $\eqref{c}$. To this end, it suffices to check the following:

\begin{itemize}
\item[(i)] $\phi_1 > \phi_0$ \qquad in $\{x \in \Omega : m(x) = -\delta_0\} \bigcap \partial (\Lambda_1 \setminus \{ x\in \Omega: m(x)>0\})$;
\item[(ii)] $\phi_1 < \phi_0$ \qquad in $\{x \in \Omega : m(x) = 0\} \bigcap \partial (\Lambda_1 \setminus \{ x\in \Omega: m(x)>0\})$.
\end{itemize}
\noindent More precisely,\\

\noindent (i): When $m(x) = -\delta_0,$ by \eqref{constraints for epsilon},
$$
e^{\epsilon \alpha (m(x) - cm_{1})} = e^{\epsilon \alpha(-\delta_0 - cm_1)} > e^{ \alpha(-\delta_0 - k)}= e^{ \alpha(m(x) - k)}.
$$
\noindent Hence, $\overline{u}_2 =  \phi_0$ in a neighborhood of $\{x \in \Omega : m(x) = -\delta_0\} \bigcap \partial (\Lambda_1 \setminus \{ x\in \Omega: m(x)>0\})$.\\

\noindent (ii): When $m(x) = 0,$ by \eqref{constraints for k},
$$
e^{\epsilon \alpha (m(x) - cm_{1})}  = e^{-\epsilon \alpha cm_{1}} < e^{ -\alpha k} = e^{ \alpha(m(x) - k)}.
$$
\noindent Hence, $\overline{u}_2 =  \phi_1$ in a neighborhood of $\{x \in \Omega : m(x) = 0\} \bigcap \partial (\Lambda_1 \setminus \{ x\in \Omega: m(x)>0\})$.\\

\noindent (Notice that $\phi_i$ are strictly increasing functions of $m(x)$. Hence (possibly making $\delta_0$ smaller) the non-differentiable regions of $\overline{u}_2$ are regular level surfaces of $m(x)$ by the implicit function theorem.)
\end{proof}

\vspace{0.5cm}

\noindent Case $(c)$: When $m(x)$ has two distinct values $0< m_1 < m_2$ on $\mathfrak{M}$ and $ m \leq 0$ at some of its critical points.\\

\noindent We first decompose $\Omega$ according to the value of $m(x)$. Write $\mathfrak{M} = \mathfrak{M}_1 \bigcup \mathfrak{M}_2,$ where $\mathfrak{M}_i = \{ x_0 \in \mathfrak{M}: m(x_0) = m_i \}$, $i=1,2.$ And define
\begin{equation*}
\mathfrak{M}_0 = \{ \text{strict local maximum points }x_0\text{ of }m(x)\text{ s.t. }m(x_0)=0  \},\\
\end{equation*}
\noindent which is possibly empty. Given any $c<1$ close to 1, define
\begin{equation*}\label{nbhpeaks1}
\begin{array}{ll}
\Gamma_1 = &\{x \in \Omega: m(x)>0 \}\\
\Lambda_1 = &\text{ The union of all connected components of }\{ x\in \Omega : m(x) > -\delta_0\}\\
            & \text{ not intersecting }\mathfrak{M}_0\\
\Gamma_2 = &\text{ The union of all connected components of }\{ x\in \Omega : m(x) > c m_1\} \\
&\text{ not intersecting }\mathfrak{M}_1\\
\Lambda_2 =&  \text{ The union of all connected components of }\{ x\in \Omega : m(x) > c^2 m_1\} \\
&\text{ not intersecting }\mathfrak{M}_1
\end{array}
\end{equation*}
\noindent where $\delta_0$ is chosen as in proof of Lemma \ref{m<0 upper lemma}. We have a partition:
$$
\Omega =  (\Omega \setminus \Lambda_1) \cup(\Lambda_1 \setminus \Gamma_1) \cup(\Gamma_1 \setminus \Lambda_2) \cup (\Lambda_2 \setminus \Gamma_2) \cup\Gamma_2 .$$

\begin{lemma}\label{main2lemma}
Given $m(x)$ satisfying \textup{\textbf{(H2)}}, \textup{\textbf{(H3)}} and \textup{\textbf{(H4)}}, and that $m(x)$ attains exactly two distinct values $0< m_1<m_2$ on $\mathfrak{M}$. For each $c<1$ close to $1$, for all $\alpha$ large, there exists an upper solution $\overline{u}_3 > 0$ in the sense of Definition $\ref{uppersol}$ such that
\begin{equation*}
\overline{u}_3(x) \leq
\left\{
\begin{array}{ll}
 e^{\alpha(m(x) - k)} \qquad & \text{ in } \Omega \setminus \Lambda_1\\
 e^{\epsilon_1 \alpha (m(x) - cm_1)} \qquad &\text{ in }  \Lambda_1 \setminus\Lambda_2,\\
 e^{\epsilon_2 \alpha (m(x) - cm_2)} \qquad &\text{ in }  \Lambda_2,
\end{array}\right.
\end{equation*}
\noindent where $0<\epsilon_i<1, k>0$ are appropriately chosen constants independent of $\alpha$.
\end{lemma}

\noindent Notice that in $ \{ x \in \Lambda_2:m(x) < cm_2\} \bigcup \{x \in  \Omega \setminus \Lambda_2: m(x) < cm_1\}$,
\begin{equation*}
\overline{u}_3 \rightarrow 0\quad \text{ as }\alpha \rightarrow \infty.
\end{equation*}
\noindent We see that in the case $m(x)$ having two distinct values $m_1 < m_2$ on $\mathfrak{M}$, Theorem $\ref{main2}$ follows as before from Lemma $\ref{main2lemma}$ and Lemma $\ref{upperabove}$.

\begin{proof}[Proof of Lemma $\ref{main2lemma}$]

\noindent Let $\phi_0 := e^{\alpha (m(x) - k)}$ and $\phi_i := e^{\epsilon_i \alpha (m(x) - cm_i)}$ ($i=1,2$), where $0<\epsilon_1<1$ is chosen to satisfy
\begin{equation}\label{constraints for epsilon_1}
 \epsilon_1 < \frac{\delta_0}{cm_1 + \delta_0},
\end{equation}
\noindent $k>0$ and $0<\epsilon_2<1$ are chosen such that
\begin{equation}\label{constraints for k_1}
0< k < \epsilon_1 c m_1,
\end{equation}
\begin{equation}\label{constraints for epsilon_2}
 0<\epsilon_{2} < \min\{\frac{\epsilon_1(c^2 m_1 - c m_1)}{c^2 m_1 - cm_{2}},1\}.
\end{equation}
\noindent We can now define $\overline{u}_3$.
\begin{equation*}
\overline{u}_3 := \left\{
\begin{array}{ll}
 \phi_0 & \text{in } \Omega \setminus \Lambda_1\\
 \phi_1 & \text{in } \Gamma_1 \setminus \Lambda_2\\
 \phi_2 & \text{in } \Gamma_2\\
 \min\{ \phi_0, \phi_1\} & \text{in } \Lambda_1 \setminus \Gamma_1\\
 \min\{ \phi_1, \phi_2 \} & \text{in } \Lambda_2 \setminus \Gamma_2\\
\end{array} \right.
\end{equation*}
\noindent It can then be proved as before that
\begin{equation*}
L\overline{u}_3\leq 0\text{ in }\Omega\quad\text{ and }\quad \frac{\partial \overline{u}_2}{\partial \nu}- \alpha \overline{u}_2 \frac{\partial m}{\partial \nu}\geq 0\text{ on } \partial \Omega
\end{equation*}
\noindent whenever they are defined. It remains to show the continuity of $\overline{u}_3$, as well as $\eqref{c}$. It suffices to show:

\begin{itemize}
\item[(i)] $\phi_0 < \phi_1$ \qquad in $\{x \in \Omega : m(x) = -\delta_0\} \bigcap \partial (\Lambda_1 \setminus \Gamma_1)$;
\item[(ii)] $\phi_0 > \phi_1$ \qquad in $\{x \in \Omega : m(x) = 0\} \bigcap \partial (\Lambda_1 \setminus \Gamma_1)$;
\item[(iii)] $\phi_1 < \phi_2$ \qquad in $\{x \in \Omega : m(x) = c^2m_1\} \bigcap \partial (\Lambda_2 \setminus \Gamma_2)$;
\item[(iv)] $\phi_1 > \phi_2$ \qquad in $\{x \in \Omega : m(x) = cm_1\} \bigcap \partial (\Lambda_2 \setminus \Gamma_2)$.
\end{itemize}
\noindent (i), (ii) can be verified following similar lines as in proof of Lemma $\ref{m<0 upper lemma}$, using $\eqref{constraints for epsilon_1}$ and $\eqref{constraints for k_1}$.\\

\noindent (iii): When $m(x) = c^2m_1$, by \eqref{constraints for epsilon_2}
$$
e^{\epsilon_{1} \alpha (m(x) - cm_{1})}  = e^{\epsilon_{1} \alpha (c^2 m_1 - cm_{1})} < e^{\epsilon_{2} \alpha (c^2 m_1 - cm_{2})} = e^{\epsilon_{2} \alpha (m(x) - cm_{2})},\text{ for }\alpha>0.
$$

\noindent (iv): When $m(x) = cm_1$
$$
e^{\epsilon_{1} \alpha (m(x) - cm_{1})}  = 1 > e^{\epsilon_{2} \alpha c (m_1 - m_{2})} = e^{\epsilon_{2} \alpha (m(x) - cm_{2})},\text{ for }\alpha>0.
$$

\end{proof}
\noindent Hence, Theorem $\ref{main2}$ is proved for the cases when $m(x)$ attains 1 or 2 values on $\mathfrak{M}$.
\end{proof}

\noindent The proof of Theorem $\ref{main3}$ is a modification of the proof in $\cite{CL}$, overcoming the difficulty caused by the local minimum and saddle points of $m(x)$. We start with the following lemma.

\begin{lemma}\label{main3lem1}
With the assumption of Theorem $\ref{main3}$, there exists $C>0$ such that
\begin{equation} \label{eqqlemma}
 u(x) \leq C e^{\alpha (m(x) - m_1)} \qquad \text{for all } x \in \Omega\text{ and all }\alpha\text{ large.}
\end{equation}
\noindent where $m_1$ is the unique value of $m(x)$ on $\mathfrak{M}$.
\end{lemma}

\begin{proof}
Consider $w=e^{(-\alpha + \epsilon)m(x)}u(x)$. Then in $\Omega$, $w$ satisfies
\begin{equation}\label{eq6}
\Delta w + (\alpha - 2\epsilon)\nabla m \cdot \nabla w - \{\epsilon (\alpha - \epsilon)|\nabla m|^2 + \epsilon \Delta m + u - m\}w =0
\end{equation}
Let $z^*=z^*(\alpha) \in \overline{\Omega} $ be such that $w(z^*) = \max_{\overline{\Omega}}w$.
Then, for $x\in \Omega$,
\begin{equation} \label{eqq0}
u(x) \leq u(z^*) e^{(-\alpha + \epsilon)(m(z^*)-m(x))}.
\end{equation}
We notice that on $\partial \Omega$,
\begin{align*}
\frac{\partial w}{\partial \nu}
&= e^{(-\alpha + \epsilon)m(x)}(\frac{\partial u}{\partial \nu} + (-\alpha + \epsilon)u\frac{\partial m}{\partial \nu})\\
&= e^{(-\alpha + \epsilon)m(x)}(\alpha u \frac{\partial m}{\partial \nu} + (-\alpha + \epsilon)u\frac{\partial m}{\partial \nu})\\
&= e^{(-\alpha + \epsilon)m(x)} \epsilon \frac{\partial m}{\partial \nu} \leq 0.
\end{align*}
Therefore by the maximum principle, no matter $z^* \in \partial \Omega$ or $\Omega$, $\nabla w (z^*) =0$ and $ \Delta w (z^*) \leq 0$. Hence, by $\eqref{eq6}$\\
\begin{equation}\label{eqq1}
\epsilon (\alpha - \epsilon) |\nabla m|^2 + \epsilon \Delta m + u \leq m \qquad \text{ at }x=z^*,
\end{equation}
\noindent and
\begin{equation} \label{eqq2}
u(z^*) \leq m(z^*) - \epsilon \Delta m(z^*).
\end{equation}
Now take $\epsilon = \max_{x_0} \{ \frac{m(x_0)}{\Delta m(x_0)} \}$, with the maximum taken over all positive saddle points and local minimum points $x_0$ of $m(x)$ such that $m(x_0)>0$. (Take $\epsilon = 1$ if it is an empty set.) Notice that $\epsilon>0$ by \textup{\textbf{(H4)}}. Then by (\ref{eqq1}), we have
\begin{equation*}
 \epsilon(\alpha - \epsilon)|\nabla m|^2 \leq m(z^*) - \epsilon \Delta m \leq |m|_\infty + \epsilon|\Delta m|_\infty ,
\end{equation*}
\noindent which implies that $ |\nabla m(z^*)| \rightarrow 0$ as $\alpha \rightarrow \infty$. Thus,
\begin{equation*}
\textup{dist}(z^*,\{ x \in \Omega: |\nabla m(x)| = 0\}) \rightarrow 0.
\end{equation*}

\noindent Next, we claim that in fact we have $\textup{dist}(z^*, \mathfrak{M})  \rightarrow 0.$\\

\noindent Assume to the contrary that there exists $\alpha_k \rightarrow \infty$, such that $z^*(\alpha_k) \rightarrow x_0$ as $k \rightarrow \infty$ where $x_0$ is a saddle point or a minimum point. Then by $\eqref{eqq2}$ and the choice of $\epsilon$,
\begin{equation*}
0 \leq u(z^*) \leq m(z^*) - \epsilon \Delta m(z^*) \rightarrow m(x_0) - \epsilon \Delta m (x_0) < 0,
\end{equation*}
which is a contradiction. Therefore, $\textup{dist}(z^*, \mathfrak{M})  \rightarrow 0$. Recalling that $m(x) \equiv m_1$ on $\mathfrak{M}$, we deduce that there exists $C>0$ such that
\begin{equation*} \label{eq7}
m_1 - m(z^*) \leq C |\nabla m(z^*)|^2, \text{ for all } \alpha \text{ large,}
\end{equation*}
since the inequality holds in a neighborhood of $\mathfrak{M}$, where $z^*$ eventually enters. Hence by $\eqref{eqq1}$ again,
\begin{equation*}
(\alpha - \epsilon) (m_1 - m(z^*)) \leq C(\alpha - \epsilon) |\nabla m(z^*)|^2 \leq C \big (\frac{m(z^*)}{\epsilon} - \Delta m(z^*)\big).
\end{equation*}
\noindent Therefore,
\begin{equation} \label{eqq3}
 (\alpha - \epsilon)(m_1 - m(z^*)) \leq C\big( \frac{m_1}{\epsilon} + \| \Delta m \|_\infty  \big)
\end{equation}
\\
And for every $x \in {\Omega}$, from $\eqref{eqq0}$,
\begin{align*}
 e^{-\alpha(m(x) - m_1)}u(x)  & \leq e^{-\alpha(m(x) - m_1)}u(z^*) e^{(\alpha - \epsilon)[m(x) - m(z^*)]}\\
             & = u(z^*) e^{\epsilon(m_1 - m(x)) + (\alpha - \epsilon)(m_1 - m(z^*))}\\
             & \leq (m_1 + \epsilon \| \Delta m \|_{\infty})e^{2 \epsilon |m|_\infty + C(\frac{m_1}{\epsilon} + \| \Delta m \|_\infty)},
\end{align*}
\noindent by $\eqref{eqq2}$ and $\eqref{eqq3}$. Since the right hand side is a constant independent of $x$ and $\alpha$, $\eqref{eqqlemma}$ is proved.

\end{proof}

\begin{proof}[Proof of Theorem $\ref{main3}$]
From $(\ref{eqqlemma})$, we see that for all $p\geq1$, $u \rightarrow 0$ in $L^p$ as $\alpha \rightarrow \infty$. For each $x_0 \in \mathfrak{M}$, fix a neighborhood $\mathfrak{U}(x_0)$ of $x_0$, by $\eqref{eqqlemma}$,
\begin{equation*}
u(x)  \leq C e^{\alpha(m(x) - m^*)} \leq C e^{\alpha(\frac{1}{2}(x-x_0)^T D^2 m(x_0)(x-x_0) + C_1|x-x_0|^3)},
\end{equation*}
\noindent where $C_1 = \|D^3 m\|_\infty /6$. Denote $M(x_0,\alpha) = \sup_\mathfrak{U}(x_0) u$, which is attained in $B_{R/\sqrt{\alpha}}(x_0)$ for $R$ sufficiently large, and all large $\alpha$ (by Theorem \ref{main1} and Lemma $\ref{main3lem1}$). Define
\begin{equation*}
W_\alpha(y) = \frac{u(x_0 + \frac{y}{\sqrt{\alpha}})}{M(x_0,\alpha)}
\end{equation*}
\noindent Then $\sup W_\alpha =1$ in $\sqrt{\alpha}\big(\mathfrak{U}(x_0) - x_0\big)$, and
$$
W_\alpha(y) \leq C e^{ \frac{1}{2} y^T D^2 m(x_0)y + \frac{C_1}{\sqrt{\alpha}}|y|^3}\leq C e^{ \frac{1}{3} y^T D^2 m(x_0)y}
$$
for all $\alpha$ large and in $\{y \in \mathbb{R}_N: x_0 + y/\sqrt{\alpha}  \in \Omega\text{, }|y| \leq \frac{-\lambda_N \sqrt{\alpha}}{6C}\}$, where $\lambda_1\leq \cdot \cdot \cdot \leq \lambda_N < 0$ are the eigenvalues of $D^2 m(x_0)$.\\

\noindent To prove $\eqref{eqq001}$, by Lemma $\ref{main3lem1}$ and the fact that $M(x_0,\alpha)$ is bounded, it suffices to show that for each $x_0\in\mathfrak{M}$
\begin{equation} \label{eq8}
\left\{
\begin{array}{ll}
W_\alpha(y) \rightarrow e^{ \frac{1}{2} y^T D^2 m(x_0)y} \text{ in every compact subset of }\mathbb{R}^N, \text{ and}\\
M(x_0,\alpha) \rightarrow 2^{N/2} m(x_0),
\end{array}\right.
\end{equation}
\noindent as $\alpha \rightarrow \infty$. $W_\alpha$ satisfies $\Delta_y W_\alpha + \overrightarrow{P} \cdot \nabla_y W_\alpha + Q W_\alpha = 0,$ where
\begin{equation*}
\overrightarrow{P} = \overrightarrow{P}(\alpha, y) = -\sqrt{\alpha} \cdot \nabla_x m \big( x_0 + \frac{y}{\sqrt{\alpha}}   \big),
\end{equation*}
and
\begin{equation*}
Q(\alpha,y) = -\Delta_x m \big( x_0 + \frac{y}{\sqrt{\alpha}} \big) - \frac{u( x_0 + \frac{y}{\sqrt{\alpha}}) - m( x_0 + \frac{y}{\sqrt{\alpha}})}{\alpha}.
\end{equation*}
The boundedness of $u$ (by $\eqref{eqqlemma}$) implies that
\begin{equation*}
\lim_{\alpha\rightarrow\infty} \overrightarrow{P}(\alpha,y) = -y^T D^2 m(x_0), \phantom{1} \lim_{\alpha\rightarrow \infty} Q(\alpha,y) = - \Delta_x m(x_0),
\end{equation*}
uniformly in any compact subset of $\mathbb{R}^2$. Hence by elliptic estimates (see $\cite{GT}$), using the fact that for each compact subset $K$ in $\mathbb{R}^N$, $W_\alpha$ is bounded in $L^p(K)$ for $p \in (1,\infty]$ and all large $\alpha$, after passing to a subsequence if necessary, as $\alpha \to \infty$, $W_\alpha$ converges to some function $W^*$ uniformly in any compact subset of $\mathbb{R}^N$, and $W^*$ must satisfy
\begin{equation}\label{W1}
\left\{
\begin{array}{ll}
\Delta_y W^* - y D^2 m(x_0) \nabla_y W^* - \Delta m(x_0) W^* = 0 &\text{ in } \mathbb{R}^N,\\
\sup_{\mathbb{R}^N} W^* (y) = 1, \phantom{2} 0\leq W^*(y) \leq C e^{ \frac{1}{3}y^T D^2 m(x_0)y} & \forall y \in \mathbb{R}^N.
\end{array}\right.
\end{equation}
Now we invoke the following lemma, the proof of which makes use of a Liouville-type result due to $\cite{BCN}$ which is formulated differently in $\cite{D}$, and will be included in the Appendix C for completeness.

\begin{lemma}\label{liouville}
If $W^* \in W_{loc}^{1,2}(\mathbb{R}^N)$ satisfies $\eqref{W1}$, then $W^* = e^{\frac{1}{2} y^T D^2 m(x_0) y}$.
\end{lemma}

\noindent The uniqueness of the limit implies that
\begin{equation}\label{sorry3}
\lim_{\alpha \rightarrow \infty} W_\alpha(y)= e^{\frac{1}{2} y^T D^2 m(x_0) y} \text{ uniformly in any compact subset of }\mathbb{R}^N.
\end{equation}
That $W^*$ attains its strict maximum at the origin and $\eqref{eqqlemma}$ implies that
\begin{equation}\label{eq9}
\lim_{\alpha \rightarrow \infty} \frac{u(x_0)}{M(x_0,\alpha)} = W^*(0) = 1.
\end{equation}

\noindent To show the second part of $\eqref{eq8}$, it remains to calculate $\displaystyle \lim_{\alpha\rightarrow\infty} u(x_0)$. In \cite{CL} it was accomplished when $m$ as a single peak via a "global" argument. Here we devise a "local" argument near each $x_0 \in \mathfrak{M}$.

\begin{lemma}\label{liminfgeq}
For each $x_0 \in \mathfrak{M}$, $ \displaystyle \liminf_{\alpha \rightarrow \infty} u(x_0) \geq 2^{N/2}m_1.$
\end{lemma}
\begin{proof}
By following the proof of Theorem $\ref{main1}$, with the same choice of test function $\psi$ and open sets $U_i$, we have for each $\eta>0$,
\begin{align*}
0  \leq& \liminf_{\alpha \rightarrow \infty} \frac{\int_{U_3} e^{\alpha[m-m_1]}(u-m)dx}{\int_{U_2} e^{\alpha[m-m_1]}\,dx}\\
   \leq& \liminf_{\alpha \rightarrow \infty} \left[ \frac{\int_{B_{R/\sqrt{\alpha}}(x_0)} e^{\alpha[m-m_1]} u \,dx}{\int_{U_2} e^{\alpha[m-m_1]}\,dx} +  \frac{\int_{U_3 \setminus B_{R/\sqrt{\alpha}}(x_0)} e^{\alpha[m-m_1]} u \,dx}{\int_{U_2} e^{\alpha[m-m_1]}\,dx} \right] \\
   &- \lim_{\alpha \rightarrow \infty} \frac{\int_{U_3} e^{\alpha[m-m_1]}m\,dx}{\int_{U_2} e^{\alpha[m-m_1]}\,dx}\\
   \leq& \liminf_{\alpha \rightarrow \infty} \left[ \frac{\int_{B_{R/\sqrt{\alpha}}(x_0)} (1+\eta) u (x_0) e^{\alpha[m(x)-m_1] + \frac{\alpha}{2} (x-x_0)^T D^2m(x_0)(x-x_0)}  \,dx}{\int_{B_{R/\sqrt{\alpha}}(x_0)} e^{\alpha[m-m_1]}\,dx} \right.\\
   & + \left. \frac{\int_{U_3 \setminus B_{R/\sqrt{\alpha}}} e^{\alpha[m-m_1]} u \,dx}{\int_{U_2} e^{\alpha[m-m_1]} \,dx}\right] - m(x_0)\\
   \leq& \liminf_{\alpha \rightarrow \infty} \left[(1+\eta)u(x_0) \frac{\int_{B_R(0)} e^{\alpha[m(x_0+ \frac{y}{\sqrt{\alpha}}) - m_1] + \frac{1}{2}y^T D^2m(x_0)y}\,dy}{\int_{B_R(0)} e^{\alpha[m(x_0+ \frac{y}{\sqrt{\alpha}}) - m_1]}\,dy}\right. \\
   &\left. \frac{\int_{\mathbb{R}^N \setminus B_{R}(0)}e^{-c_1 |y|^2\,dy}}{\int_{B_R(0)} e^{-c_2|y|^2}\,dy}\right] - m(x_0)\\
   \leq& (1+\eta)\left[\liminf_{\alpha \rightarrow \infty} u(x_0) \right] (2^{-\frac{N}{2}} + \eta) + \eta - m(x_0)
\end{align*}
The third inequality follows from \eqref{sorry3}, \eqref{eq9} and the Lebesgue Dominated Convergence. In the fourth inequality, we applied the change of coordinates $x = x_0 + \frac{y}{\sqrt{\alpha}}$ and that there exists $c_1,c_2>0$ such that $c_1|y|^2 \leq m_1 - m(x) \leq c_2|y|^2$ (which are consequences of the nondegeneracy of $m$). The last line follows by taking $R>0$ sufficiently large and that
$$
\lim_{\alpha\rightarrow\infty} \alpha[m(x_0 + \frac{y}{\sqrt{\alpha}}) - m_1] = \frac{1}{2} y^T D^2m(x_0) y
$$
\noindent uniformly in compact subsets of $\mathbb{R}^N$. Finally, the lemma is proved by letting $\eta \to 0^+$
\end{proof}
\noindent Next, we claim that
\begin{claim}\label{claim}
$\displaystyle \lim_{\alpha\rightarrow\infty} \sum_{x_0 \in \mathfrak{M}} \int_{\mathbb{R}^N} e^{\frac{1}{2} y^T D^2m(x_0) y} \,dy \left[u(x_0) ^2 - 2^{N/2}m_1 u(x_0) \right] = 0$
\end{claim}
\begin{proof}[Proof of Claim \ref{claim}]
Integrate \eqref{eq1} over $\Omega$, we have
\begin{align*}
0 = & \int_\Omega (u^2 - um)\,dx\\
  =& \left\{\int_{\cup_{\mathfrak{M}}B_{R/\sqrt{\alpha}}(x_0)}  + \int_{\cup_{\mathfrak{M}}B_{r_0}(x_0) \setminus B_{R/\sqrt{\alpha}}(x_0)}  + \int_{\Omega \setminus \cup_{\mathfrak{M}} B_{r_0}(x_0)}   \right\} (u^2-um)\,dx\\
  =& \sum_{x_0 \in \mathfrak{M}}\left[ \int_{B_{R/\sqrt{\alpha}}(x_0)}(u^2-um)\,dx + C\int_{B_{r_0}(x_0) \setminus B_{R/\sqrt{\alpha}}(x_0)} e^{\alpha [m(x)-m_1]}\,dx\right]\\
  & + O(e^{-\gamma \alpha}).
\end{align*}
\noindent by Theorem \ref{main2} and Lemma \ref{main3lem1}. Multiply by $\alpha^{\frac{N}{2}}$ and changing coordinates $x=x_0+ \frac{y}{\sqrt{\alpha}}$, we see that
$$
0 = \sum_{x_0\in\mathfrak{M}} \int_{B_R(0)}(u^2 -um)(x_0 + \frac{y}{\sqrt{\alpha}})\,dy + O(\int_{\mathbb{R}^N\setminus B_R(0)}e^{-c_1 |y|^2}\,dy) + O(\alpha^{\frac{N}{2}} e^{-\gamma \alpha}).
$$
By \eqref{sorry3} and \eqref{eq9}, for each $R>0$ large, there exists $\alpha_0$ such that for any $\alpha \geq \alpha_0$,
\begin{align*}
0=& \sum_{x_0\in\mathfrak{M}} \int_{B_R(0)} \left[ u^2(x_0) e^{y^T D^2m(x_0) y} - u(x_0) m(x_0) e^{\frac{1}{2} y^T D^2m(x_0)y}  \right]\,dy + o(1)\\
  &+ O(\int_{\mathbb{R}^N \setminus B_R(0)} e^{-c_1|y|^2}\,dy)\\
 =& \sum_{x_0\in\mathfrak{M}} \int_{\mathbb{R}^N}\left[ u^2(x_0) e^{y^T D^2m(x_0)y} - u(x_0)m(x_0) e^{\frac{1}{2} y^T D^2m(x_0)y}\right]\,dy + o(1)\\
  &+ O(\int_{\mathbb{R}^N \setminus B_R(0)} e^{-c_3|y|^2}\,dy).
\end{align*}
\noindent where $\displaystyle \lim_{\alpha \rightarrow \infty} o(1) = 0$. Now take $\alpha \to \infty$ and then $R \to \infty$, we have the desired result.

\end{proof}
\noindent Lemma \ref{liminfgeq} and Claim \ref{claim} implies the second part of $\eqref{eq8}$. This concludes the proof of Theorem \ref{main3}
\end{proof}

\section {Proof of Theorem $\textbf{\ref{main4}}$}
\label{sec:3}

As before, assume for simplicity $d_1=1$.

\begin{proof}[Proof of Theorem $\ref{main4}$]
Notice that $(U_\alpha,V_\alpha)$ satisfies
\begin{equation}\label{main4.1}
\left\{
\begin{array}{rl}
\nabla\cdot(\nabla U - \alpha U \nabla m) + U ( m - U ) &= UV > 0  \qquad in\phantom{1}\Omega,\\
d_2\Delta V + V( m  - V) &= UV > 0 \qquad in\phantom{1} \Omega, \\
\frac{\partial U}{\partial \nu} - \alpha U \frac{\partial m}{\partial \nu} = \frac{\partial V}{\partial \nu} &= 0 \qquad \qquad on\phantom{1} \partial \Omega.
\end{array}\right.
\end{equation}
By method of upper and lower solutions, $0 < U_\alpha \leq u$ and $0<V_\alpha \leq \theta_{d_2}$. $\eqref{main41}$ follows from the same argument as in proof of Theorem $\ref{main1}$, using the inequality $V_\alpha \leq \theta_{d_2}$. That $U_\alpha$ converges to $0$ away from the positive local maximum points of $m(x)$ follows from the corresponding property of $u$.\\

\noindent Now, assume $m \equiv m_1$ on the set of its local maximum points.

\begin{lemma}\label{main4lem1}
If \textup{\textbf{(H2)}}, \textup{\textbf{(H3)}} and \textup{\textbf{(H4)}} hold, and $m(x)$ is constant on its local maximum points, then there exists $C_2>0$ such that
\begin{equation*}
 U_\alpha(x) \leq C_2 e^{\alpha (m(x) - m_1)} \qquad \text{for all } x \in \Omega \text{ and all }\alpha \text{ large}.
\end{equation*}
\end{lemma}
\noindent Lemma $\ref{main4lem1}$ follows from Lemma $\ref{main3lem1}$ and the fact that $0 <U_\alpha \leq u$.\\

\noindent For some $\alpha_0$ large, $\int_\Omega [m - C_2 e^{\alpha_0 ( m(x) - m_1 )}] > 0$ and by a claim on P. 498 in $\cite{CCL}$, there exists a positive solution $V_0$ of
\begin{equation*}\label{main4.2a}
\left\{
\begin{array}{ll}
d_2 \Delta V_0 + V_0 ( m - C_2 e^{\alpha_0 ( m(x) - m_1 )} - V_0) = 0 & \phantom{1} in \phantom{1} \Omega,\\
\frac{\partial V_0}{\partial \nu}  = 0 & \phantom{1} on \phantom{1}  \partial \Omega.
\end{array}
\right.
\end{equation*}
then for all $\alpha \geq \alpha_0$,

\begin{equation*}
\left\{
\begin{array}{ll}
\Delta V_0 + V_0 ( m - U_\alpha - V_0) &\geq 0 \qquad  in\phantom{1}\Omega,\\
\frac{\partial V_0}{\partial \nu} & = 0 \qquad  on\phantom{1} \partial \Omega.
\end{array}\right.
\end{equation*}

\noindent Therefore, $V_0$ is a lower solution of the second equation of $\eqref{main4.1}$ for $V_\alpha$, and,
\begin{equation} \label{eq11}
\theta_{d_2} \geq V_\alpha \geq V_0 > 0 \qquad \text{        for all }\alpha \geq \alpha_0.
\end{equation}

\noindent By Lemma $\ref{main4lem1}$, $U_\alpha \rightarrow 0$ in $L^p$ for any $p>1$. By second equation in $(\ref{main4.1})$, $\eqref{eq11}$, and elliptic estimates and uniqueness, $V \rightharpoonup \theta_{d_2}$ weakly in $W^{2,p}(\Omega)$ in any $p>1$ hence strongly in $C^{1,\beta}(\overline{\Omega})$ for any $\beta \in (0,1)$. This proves $\eqref{main4.2}$.\\

\noindent Fix $x_0 \in \mathfrak{M}$ and let $\widetilde{W}_\alpha(y) = \frac{U_\alpha(x_0 + y/\sqrt{\alpha})}{M(x_0,\alpha)}$, where $M(x_0,\alpha) = \sup_{B_{r_0}(x_0)}U_\alpha$ for some small $r_0>0$. ($M(x_0,\alpha)$ is independent of the choice of $r_0$ by $\eqref{main41}$ and Lemma $\ref{main4lem1}$.) As in proof of Theorem $\ref{main3}$, notice that $\widetilde{W}_\alpha(y) \rightarrow \widetilde{W}^*(y)$ as $\alpha \rightarrow \infty$ uniformly for $y$ in compact sets in $\mathbb{R}^N$ where $\widetilde{W}^*$ satisfies
\begin{equation*}
\Delta_y \widetilde{W}^* - y D^2 m(x_0) \nabla_y \widetilde{W}^* - \Delta m(x_0)  \widetilde{W}^* = 0 \qquad in \phantom{1} \mathbb{R}^N.
\end{equation*}

\noindent Also similar as in proof of Theorem $\ref{main3}$,
\begin{equation}\label{sorry1}
\lim_{\alpha\rightarrow \infty} \widetilde{W}_\alpha(y) = W^*(y) = e^{\frac{1}{2} y^T D^2m(x_0) y}\text{ on compact sets in }\mathbb{R}^N
 \end{equation}
and $\displaystyle \lim_{\alpha\rightarrow \infty} \frac{U(x_i)}{M(x_i,\alpha)} = 1$. Now, by arguments in the proof of Theorem $\ref{main1}$, we have
\begin{equation*}\label{sorry5}
\liminf_{\alpha\rightarrow\infty} \frac{\int_{U_3} e^{\alpha m} (U_\alpha + V_\alpha - m)}{\int_{U_2} e^{\alpha m}} \geq 0.
\end{equation*}
\noindent Then Lemma $\ref{main4lem1}$, $\eqref{main4.2}$ and $\eqref{sorry1}$ implies, for each $x_0 \in \mathfrak{M}$,
\begin{equation}\label{sorry2}
\liminf_{\alpha \rightarrow \infty} U_\alpha(x_0) \geq 2^{N/2}(m_1- \theta_{d_2}(x_0))
\end{equation}
\noindent By integrating the first equation of $\eqref{main4.1}$ over $\Omega$, we have $\int_\Omega U_\alpha (m-U_\alpha-V_\alpha)dx = 0$. And by similar arguments in proving Claim \ref{claim}, we have
\begin{equation}\label{33}
0 = \lim_{\alpha \rightarrow \infty} \sum_{x_0 \in \mathfrak{M}} \int_{\mathbb{R}^N}e^{\frac{1}{2}y^T D^2m(x_0)y}dy[U_\alpha(x_0)^2 - 2^{N/2}(m_1-\theta_{d_2}(x_0)) U_\alpha(x_0)].
\end{equation}
\noindent Finally, $\displaystyle \lim_{\alpha\rightarrow \infty} U_\alpha(x_0) = 2^{N/2}(m_1 - \theta_{d_2}(x_0))$ follows from $\eqref{sorry2}$ and \eqref{33}.
\end{proof}

\section {Concluding Remarks}
\label{sec:4}

\noindent In this paper, the existence of concentration phenomena in the globally stable steady state $u(x)$ of $\eqref{eq0}$ is proved for $m(x)$ which has finitely many local maximum points. Furthermore, the concentration set is shown to be the set of positive local maximum points of $m(x)$. The situation when $m(x)$ contains local maximums that are not strict is however, completely open. It is possible that $u$ would concentrate on some higher dimensional sets.\\

\noindent In this paper, the limiting profile is obtained in the special case when the resource function $m$ has equal peaks. Based on the estimates established in this paper, a special method is introduced to determine the limiting profile for $m$ with peaks of different heights in $\cite{LN}$. However, the method only works for $N=1$. For $N\geq 2$, very recently the limiting profile has been found by the author. This will be published in a forthcoming paper.\\

\noindent We learnt recently that in $\cite{BL}$, a lower solution for $\eqref{eq1}$ can be constructed at each $x_0 \in \mathfrak{M}$ which gives an alternative proof for the existence of peaks on $\mathfrak{M}$.\\

\noindent We also remark that the assumptions on $m(x)$ in $\{x\in\Omega: m(x)<0  \}$ can be weakened substantially. In fact, instead of \textup{\textbf{(H2)}}, \textup{\textbf{(H3)}} and \textup{\textbf{(H4)}}, we only need to assume that there exists $\delta>0$, such that the followings hold.
\begin{description}
\item[(H2')] $ \frac{\partial m}{\partial \nu} \leq 0 $ on $\{x\in\partial{\Omega}: m(x) \geq -\delta\}.$
\item[(H3')] $m(x)$ has finitely many local maximum points in $\{x\in\overline{\Omega}: m(x) \geq -\delta\}$, all being strict local maxima and are located in the interior of $\Omega$.
\item[(H4')] If $x_0\in \overline{\Omega}$ satisfies $m(x_0)\geq -\delta$ and is a local minimum or a saddle point of $m(x)$, then $\Delta m (x_0) > 0$
\end{description}

\noindent Finally, notice that although we have set the diffusion coefficient $d,d_1=1$ for simplicity, the results proved in this paper hold true for any $d,d_1>0$, as stated in Section 1.

\section {Appendix A}\label{sec:A}

\noindent Denote $\theta_d$ to be the unique positive solution to
\begin{equation*}
\left\{
\begin{array}{rl}
d\Delta \theta + \theta ( m - \theta) &= 0 \qquad  in\phantom{1}\Omega,\\
\frac{\partial \theta}{\partial \nu} & = 0 \qquad on\phantom{1} \partial \Omega.
\end{array}\right.
\end{equation*}
\noindent The existence part is standard. (See, e.g. P. 498 in \cite{CCL}.) Also, it is known that (Prop. 3.16 of $\cite{CC3}$)
\begin{equation}\label{theta converges to m+}
\lim_{d\rightarrow 0^+} \theta_d = m^+, \quad \text{ uniformly in }\Omega,\tag{A1}
\end{equation}
\noindent where $m^+(x) : = \max\{m(x),0  \}$.\\

\noindent Here we shall prove that if $x_0 \in \Omega$ is a positive strict local maximum point of $m$ and $\Delta m(x_0)<0$, then $m(x_0) - \theta_{d}(x_0)>0$ for all $d>0$ sufficiently small.

\begin{remark}[\cite{L2}]
When $d$ is not small, there are counter examples showing that the conclusion is not true in general for $x_0 \in \mathfrak{M}$ other than the global maximum point(s).
\end{remark}

\noindent First we show that $m(x_0) \geq \theta_{d}(x_0)$. Assume now to the contrary that for some positive strict positive local maximum point $x_0$ of $m(x)$, for some sequence $d_i \rightarrow 0$,
\begin{equation*}
\theta_{d_i} (x_0) > m(x_0)>0.
\end{equation*}
Now, $x_0 \in \{x\in\Omega: \theta_{d_i} (x) > m(x_0)\}$ for all $i$. Denote by $U_i$ the connected component of $\{x\in\Omega: \theta_{d_i} (x) > m(x_0)\}$ that contains $x_0$, then $U_i \neq \emptyset$ and $\Delta \theta_{d_i} = \theta_{d_i}(\theta_{d_i} - m) \geq 0$ in $U_i$. i.e. $\theta_{d_i}$ is subharmonic in $U_i$. Now for $d_i$ sufficiently small, by (A1), $U_i$ is compactly contained in a neighborhood of $x_0$. In particular, $\theta_{d_i} > m(x_0)$ in $U_i$ and $\theta_{d_i}(x) = m(x_0)$ on $\partial U_i$. This contradicts the property of subharmonic functions. Therefore, $m(x_0) \geq \theta_{d}(x_0)$ for all $d>0$ sufficiently small.\\

\noindent Now assume there exists a sequence $d_i \to 0$ such that $\theta_{d_i}(x_0) = m(x_0)$. We claim that
\begin{claim}\label{claim2}
$\nabla \theta_{d_i}(x_0) = 0$ for all $i$ sufficiently large.
\end{claim}
\noindent Otherwise there exists $x_i \to x_0$ such that $\theta_{d_i}(x_i) > m(x_0)$ and a contradiction can be reached by previous arguments by choosing a horizontal hyperplane.\\

\noindent Now since $\theta_{d_i}(x_0) = m(x_0)$, $\nabla \theta_{d_i}(x_0) = \nabla m(x_0)$ and $\Delta \theta_{d_i} =0 > \nabla m(x_0)$, there exists $x_i \to x_0$ such that $\theta_{d_i}(x_i) > m(x_i)$. (Since otherwise the mean curvature of the surface defined by $\theta_{d_i}$ in $\mathbb{R}^{N+1}$ at $x_0$, which is a multiple of $\Delta \theta_{d_2} (x_0)$, would not be not equal to $0$.) Now fix a neighborhood $U_0$ of $x_0$, and a (slightly tilted) hyperplane $\Sigma_i:L(\mathbb{R}^N,\mathbb{R})$ such that
\begin{equation*}
\theta_{d_i}(x_i) > \Sigma_i(x_i)\text{ and }\Sigma_i(x) > m(x)\text{ in }U_0.
\end{equation*}
By (A1), $\theta_{d_i} \to m$ uniformly on $\partial U_{0}$ while $\min_{\partial{U}_0} \{\Sigma_i(x) - m(x)  \}\geq c>0$ for some constant $c$ independent of $i$. This implies that there is some $U_i \neq \emptyset$ such that
\begin{equation*}
\left\{\begin{array}{ll}
\Delta \theta_{d_i} = \theta_{d_i}(\theta_{d_i} -m)\geq 0 \text{ in }U_i\\
\theta_{d_i} > \Sigma_i \text{ in } U_i, \quad
\theta_{d_i} = \Sigma_i \text{ on }\partial U_i
\end{array}\right.
\end{equation*}
which again contradicts the fact that $\theta_{d_i}$ is subharmonic in $U_i$.

\section {Appendix B}\label{sec:B}
\label{sec:5}
Here we discuss the proof of the general case of Theorem $\ref{main2}$. Recall
\begin{equation*}
\mathfrak{M}= \{ \text{ positive strict local maximum points of }m(x)\text{ in }\Omega\phantom{1}\}.
\end{equation*}
By \textup{\textbf{(H3)}}, $m(x)$ has finitely many local maximum points. Let $0<m_1<m_2<\cdot \cdot \cdot <m_{n_0}$ be the distinct values of $m(x)$ on $\mathfrak{M}$. Decompose
\begin{equation*}
\mathfrak{M} = \bigcup_{i=1}^{n_0} \mathfrak{M}_i,
\end{equation*}
\noindent where $\mathfrak{M}_i = \{ x_0\in \mathfrak{M}: m(x_0) = m_i \}$. And let
\begin{equation*}
\mathfrak{M}_0 :=\{ \text{local maximum points }x_0\text{ of }m(x) \text{ s.t. } m(x_0) = 0 \},
\end{equation*}
\noindent which is possibly empty. For each $c<1$, close to $1$. Define $\delta_0$ as in the proof of Lemma \ref{m<0 upper lemma}. Decompose $\Omega$ according to the value of $m(x)$:
\begin{align*}
\Gamma_{1} = & \{x\in \Omega: m(x) > 0  \}\\
\Lambda_{1}=  & \text{Union of connected components of } \{x\in \Omega: m(x) > -\delta_0  \}\\
               &\text{ not intersecting } \mathfrak{M}_{0},\\
\Gamma_{i} = & \text{Union of connected components of } \{x\in \Omega: m(x) > cm_{i-1}  \}\\
               &\text{ not intersecting } \mathfrak{M}_{i-1},\\
\Lambda_{i}=  & \text{Union of connected components of } \{x\in \Omega: m(x) > c^2 m_{i-1}  \}\\
               &\text{ not intersecting } \mathfrak{M}_{i-1},
\end{align*}
\noindent for $i=2,...,n_0$. Notice that $\Lambda_i \supseteq \Gamma_i \supseteq \Lambda_{i+1} \supseteq \Gamma_{i+1}$. Define
\begin{equation*}
\overline{u}(x) = \left\{
\begin{array}{ll}
 e^{\epsilon_{n_0} \alpha (m(x) - cm_{n_0})} & \text{ in } \Gamma_{n_0}\\
 e^{\epsilon_{i} \alpha (m(x) - cm_{i})} & \text{ in } \Gamma_{i} \setminus \Lambda_{i+1} \\
                                        & \text{ for }i=1,...,n_0 -1\\
 e^{\alpha(m(x)-k) }                      & \text{ in } \Omega \setminus \Lambda_{1}\\
 \min\{e^{\epsilon_{i} \alpha (m(x) - cm_{i})},e^{\epsilon_{i+1} \alpha (m(x) - cm_{i+1})} \} & \text{ in } \Lambda_{i+1} \setminus \Gamma_{i+1}\\
                                        & \text{ for }i=1,...,n_0 -1\\
 \min\{e^{\alpha(m(x)-k) },e^{\epsilon_{1} \alpha (m(x) - cm_{1})}\} & \text{ in } \Lambda_{1} \setminus \Gamma_1.
\end{array} \right.
\end{equation*}
\noindent where $0<\epsilon_i<1, k>0$ are constants chosen such that
\begin{equation*}
 \epsilon_1 < \frac{\delta_0}{c m_1 + \delta_0 },\quad 0< k < \epsilon_1 c m_1 , \text{ and}
\end{equation*}
\begin{equation*}
0<\epsilon_{i+1} < \min\{\frac{\epsilon_i(c^2 m_i - c m_i)}{c^2 m_i - cm_{i+1}},1\}, \quad \text{for } i=1,\cdot\cdot\cdot,n_0 -1.
\end{equation*}

\noindent Then, we have
\begin{lemma} \label{main2lemma1}
Given $m(x)$ satisfying \textup{\textbf{(H2)}}, \textup{\textbf{(H3)}} and \textup{\textbf{(H4)}}. For every $c<1$ sufficiently close to $1$, $\overline{u}> 0$ is an upper solution to $\eqref{eq1}$ according to Definition $\ref{uppersol}$.
\end{lemma}
\noindent The proof of Lemma $\ref{main2lemma1}$ is similar to that of Lemma $\ref{main2lemma}$ and Lemma $\ref{m<0 upper lemma}$ and is omitted.\\

\noindent Notice that the full statement of Theorem $\ref{main2}$ follows from the above lemma and Lemma $\ref{upperabove}$.

\section {Appendix C}\label{sec:C}

\noindent Next, we shall prove Lemma $\ref{liouville}$. We first state and prove the following Liouville-type theorem which is due to $\cite{BCN}$, following the formulation in $\cite{D}$.
\begin{theorem}\label{thmdu}
Let $\sigma \in L^\infty_{loc}(\mathbb{R}^N)$ be a positive function. Assume that $\Phi \in W^{1,2}_{loc}(\mathbb{R}^N)$ satisfies in the weak sense
\begin{equation}\label{7.26}
\Phi \text{\textup{ div}}(\sigma^2 \nabla \Phi) \geq 0\text{ in }\mathbb{R}^N, \tag{C1}
\end{equation}
\noindent and for some $C>0$ and every $R>1$,
\begin{equation}\label{7.27}
\int_{B_R(0)} (\sigma \Phi)^2dx \leq C R^2. \tag{C2}
\end{equation}
\noindent Then $\Phi$ is a constant.
\end{theorem}

\begin{proof}[Proof of Theorem $\ref{thmdu}$]
\noindent From $\eqref{7.26}$ we deduce, for any smooth function $\psi$,
\begin{equation}\label{7.28}
\textup{div}(\Phi \psi^2 \sigma^2 \nabla\Phi) \geq \psi^2 \sigma^2 |\nabla \Phi|^2 + 2 \Phi \psi \sigma^2 \nabla \psi \cdot \nabla \Phi.\tag{C3}
\end{equation}
\noindent Let $\zeta$ be a $C^\infty$ function on $[0,\infty)$ with $0 \leq \zeta(t) \leq 1$ and $\zeta(t) =1$ for $0 \leq t \leq 1$, $\zeta(t) = 0 $ for $t \geq 2$. For $R>0$ and $x \in \mathbb{R}^N$ set $\zeta_R (x) = \zeta(|x|/R).$\\
\noindent Taking $\psi = \zeta_R$ in $\eqref{7.28}$ and integrating over $\mathbb{R}^N$, we find, by the divergence theorem,
\begin{align*}
\int_{\mathbb{R}^N} \zeta^2_R \sigma^2 |\nabla \Phi|^2 dx &\leq 2\bigg|\int_{\mathbb{R}^N} \sigma^2 \zeta_R \Phi \nabla \zeta_R \cdot \nabla \Phi dx\bigg|\\
            &\leq 2 \bigg[ \int_{R<|x|<2R} \sigma^2 \zeta^2_R |\nabla\Phi|^2dx \bigg]^{1/2} \bigg[ \int_{\mathbb{R}^N}\sigma^2 \Phi^2 |\nabla \zeta_R|^2 dx  \bigg]^{1/2}.
\end{align*}
\noindent By $\eqref{7.27}$ and the definition of $\zeta_R$, we can find $C_1 > 0$ such that
\begin{equation*}
\int_{\mathbb{R}^N}\sigma^2 \Phi^2 |\nabla \zeta_R|^2 dx \leq C_1,
\end{equation*}
\noindent Therefore
\begin{equation}\label{7.29}
\int_{\mathbb{R}^N} \zeta^2_R \sigma^2 |\nabla \Phi|^2 dx \leq 2 \sqrt{C_1} \bigg[ \int_{R<|x|<2R} \zeta^2_R \sigma^2  |\nabla \Phi|^2 dx  \bigg]^{1/2}.\tag{C4}
\end{equation}
\noindent This implies that
\begin{equation*}
\int_{\mathbb{R}^N} \zeta^2_R \sigma^2 |\nabla \Phi|^2 dx\leq 4 C_1,
\end{equation*}
\noindent and hence, letting $R \rightarrow \infty$ in $\eqref{7.29}$ we obtain
\begin{equation*}
\int_{\mathbb{R}^N} \sigma^2 |\nabla \Phi|^2 dx =0.
\end{equation*}
\noindent This implies $|\nabla \Phi| \equiv 0$ $a.e.$ Hence $\Phi$ is a constant.
\end{proof}

\begin{proof}[Proof of Lemma $\ref{liouville}$]
\noindent Given $W^*$ satisfying $\eqref{W1}$, we want to show that $W^* = e^{\frac{1}{2}y^T D^2 m(x_0) y}$.\\

\noindent First we make the transformation $W^* = e^{-\frac{1}{2}y^T D^2 m(x_0) y}\Phi$. By $\eqref{W1}$, we see that $\Phi$ satisfies
\begin{equation*}
\left\{
\begin{array}{ll}
\text{div}( e^{\frac{1}{2}y^T D^2 m(x_0) y} \nabla \Phi) = 0 \quad \text{ in }\mathbb{R}^N ,\\
0 < \Phi \leq K_3 e^{-\frac{1}{6} y^T D^2 m(x_0) y}\text{, } \\
\sup_{\mathbb{R}^N} \Phi(y)e^{\frac{1}{2}y^T D^2 m(x_0) y} = 1.
\end{array}\right.
\end{equation*}
\noindent It remains to show that $\Phi$ is a constant. By Theorem $\ref{thmdu}$, it suffices to show that for some $C>0$ and every $R>1,$
\begin{equation}\label{W2}
\int_{B_R(0)} e^{\frac{1}{2} y^T D^2 m(x_0) y} \Phi^2dx \leq C R^2.\tag{C5}
\end{equation}
\noindent By noticing that the integrand can be dominated by
\begin{equation*}
e^{\frac{1}{2} y^T D^2 m(x_0) y} \Phi^2 \leq K^2_3 e^{\frac{1}{6} y^T D^2 m(x_0) y},
\end{equation*}
we have immediately that $\eqref{W2}$ is true. Hence the theorem is proved.
\end{proof}


\begin{thebibliography}{99}
\bibitem{CC3} R.S. Cantrell and C. Cosner:
Spatial Ecology via Reaction-Diffusion Equations, Wiley Series in Mathematical and Computational Biology, (2003).


\bibitem{LE} S. Levin:
Population Models and Community Structure in Heterogeneous Environments in Mathematical Ecology T.G. Hallam and S. Levin, Eds., Biomathematics 17, Springer-Verlag, Berlin, (1986).


\bibitem{O} A. Okubo:
Diffusion and Ecological Problems: Mathematical Models, Biomathematics 10, Springer-Verlag, Berlin, (1980).


\bibitem{SK} J.G. Skellam:
\emph{Random Dispersal in Theoretical Populations}, Biometrika 38, 196-218 (1951).


\bibitem{CC2} R.S. Cantrell and C. Cosner:
\emph{The effects of spatial heterogeneity in population dynamics}, J. Math. Biology 29, 315--338, (1991).

\bibitem{BC} F. Belgacem and C. Cosner:
\emph{The effects of dispersal along environmental gradients on the dynamics of populations
in heterogeneous environment}, Canad. Appl. Math. Quart. 3, 379-397 (1995).

\bibitem{B} F. Belgacem: Elliptic Boundary Value Problems with Indefinite Weights: Variational Formulations of the Principal Eigenvalue and Applications, Pitman Res. Notes Math. Ser., Vol. 368, Longman, Harlow, (1997).

\bibitem{CL2} C. Cosner and Y. Lou:
\emph{Does movement toward better environments always
benefit a population?}, J. Math. Anal. Appl. 277, 489-503 (2003).

\bibitem{CCL} R.S. Cantrell, C. Cosner and Y. Lou:
\emph{Advection-mediated coexistence of competing species}, Proc. Roy. Soc. Edinburgh Sect. A 137, 497-518 (2007).

\bibitem{L} Y. Lou:
\emph{Some challenging mathematical problems in evolution of dispersal and population dynamics}, Lecture Notes in Math. 1922, Springer, Berlin, 171-205 (2008).



\bibitem{CL} X. Chen and Y. Lou:
\emph{Principal eigenvalue and eigenfunctions of an elliptic operator with large advection and its application to a competition model}, Indiana Univ. Math. J. 57, 627-657 (2008).

\bibitem{CCL2} R.S. Cantrell, C. Cosner and Y. Lou:
\emph{Movement toward better environments and the evolution of rapid diffusion}, Math. Biosci. 204, 199-214 (2006).


\bibitem{DHMP} J. Dockery ,V. Hutson, K. Mischaikow, M. Pernarowski:
\emph{The evolution of slow dispersal rates: A reaction-diffusion model}, J. Math. Biol. 37, 61-83 (1998).

\bibitem{S} D.H. Sattinger:
\emph{Monotone methods in nonlinear elliptic and parabolic boundary value problems}, Indiana Univ. Math. J. 21, 979-1000 (1972).

\bibitem{D} Y. Du:
Order Structure and Topological Methods in Nonlinear Partial Differential Equations, Vol. 1, World Scientific, (2006).

\bibitem{GT} D. Gilbarg and N.S. Trudinger:
Elliptic Partial Differential Equations of Second Order, Grundlehren der Mathematischen Wissenschaften 224, Springer-Verlag, Berlin, (1983).

\bibitem{BCN} H. Berestycki, L. Caffarelli, and L. Nirenberg:
\emph{Further qualitative properties for elliptic equations in unbounded domains}, Ann. Scuola Norm. Sup. Pisa Cl. Sci. XXV, 69-94 (1997).

\bibitem{LN}K.Y. Lam and W.M. Ni:
\emph{Limint profiles of semilinear elliptic equations with large advection in population dynamics},  Discrete Contin. Dyn. Syst., Vol. 28, No. 3 , 1051 - 1067 (2009).

\bibitem{BL} A. Bezugly and Y. Lou:
\emph{Reaction-diffusion models with large advection coefficients}, submitted to Appl. Anal. (2009).


\bibitem{L2} Y. Lou:
\emph{Private communication} (2009)



\end{thebibliography}
\end{document}